\documentclass[12pt,reqno]{amsart}
\usepackage{graphicx}
\usepackage{amsmath}
\usepackage{amssymb}
\usepackage{amscd}
\usepackage{mathrsfs}
\usepackage[all,cmtip]{xy}
 \usepackage{color}
\newtheorem{theorem}{Theorem}

\newtheorem{theoremc}{Theorem}

\newtheorem{rk}[theoremc]{Remark\!}

\newtheorem{lem}[theorem]{Lemma}

\renewcommand\1{{\bf 1}}

\newcommand\com[1]{}

\newcommand\Cc{{\let\mathcal\mathscr\mathcal C}}

\newcommand\D{{\mathcal D}}

\newcommand\g{{\frak g}}

\newcommand\La{\Lambda}

\newcommand\op[1]{\mathop{\rm #1}\nolimits}
\newcommand\ot{\otimes}
\newcommand\p{\partial}

\newcommand\R{{\mathbb R}}

\newcommand\vp{\varphi}
\newcommand\we{\wedge}

\newcommand{\sym}{\mathfrak{g}}

\begin{document}

 \title[Differential Invariants of self-dual structures]{Differential Invariants of \\ Self-Dual conformal structures}
 \author{Boris Kruglikov, Eivind Schneider}
 \date{}
\address{\hspace{-17pt} Institute of Mathematics and Statistics, NT-faculty, University of Troms\o, Troms\o\ 90-37, Norway.\newline
E-mails: {\tt boris.kruglikov@uit.no, eivind.schneider@uit.no}. }

 \vspace{-14.5pt}
 \begin{abstract}
We compute the quotient of the self-duality equation for conformal metrics by the action of the
diffeomorphism group. We also determine Hilbert polynomial, counting the number of independent scalar
differential invariants depending on the jet-order, and the corresponding Poincar\'e function.
We describe the field of rational differential invariants separating generic orbits of the diffeomorphism
pseudogroup action, resolving the local recognition problem for self-dual conformal structures.
 \end{abstract}

 \maketitle

 \section*{Introduction}

Self-duality is an important phenomenon in four-dimensional differential geometry that has
numerous applications in physics, twistor theory, analysis, topology and integrability theory.
A pseudo-Riemannian metric $g$ on an oriented four-dimensional manifold $M$ determines the Hodge operator
$*:\La^2TM\to\La^2TM$ that satisfies the property $*^2=\1$ provided $g$ has the Riemannian or split signature.
In this paper we restrict to these two cases, ignoring the Lorentzian signature.

The Riemann curvature tensor splits into $O(g)$-irreducible pieces $R_g=\op{Sc}_g+\op{Ric}_0+W$,
where the last part is the Weyl tensor \cite{B} and
$O(g)$ is the orthogonal group of $g$. In dimension 4, due to exceptional isomorphisms
$\mathfrak{so}(4)=\mathfrak{so}(3)\oplus\mathfrak{so}(3)$,
$\mathfrak{so}(2,2)=\mathfrak{so}(1,2)\oplus\mathfrak{so}(1,2)$,
the last component splits further $W=W_++W_-$,
where $*W_\pm=\pm W_\pm$. Metric $g$ is called self-dual if $*W=W$, i.e.\ $W_-=0$.
This property does not depend on conformal rescalings of the metric $g\to e^{2\vp}g$,
and so is the property of the conformal structure $[g]$.

Since the space of $W_-$ has dimension 5, and the conformal structure has 9 components in 4D,
the self-duality equation appears as an underdetermined system of 5 PDE on 9 functions of 4 arguments.
This is however a misleading count, since the equation is natural, and the diffeomorphism group
acts as the symmetry group of the equation. Since $\op{Diff}(M)$ is parametrized by 4 functions of
4 arguments, we expect to obtain a system of 5 PDE on $5=9-4$ functions of 4 arguments.

This $5\times 5$ system is determined, 
but it has never been written explicitly. There are two approaches to eliminate the gauge freedom.

One way to fix the gauge is to pass to the quotient equation that is obtained
as a system of differential relations (syzygies) on a generating set of differential invariants.
By computing the latter for the self-dual conformal structures we write the quotient equation
as a nonlinear $9\times9$ PDE system, which is determined but complicated to investigate.

Another approach is to get a cross-section or a quasi-section to the orbits of the pseudogroup
$G=\op{Diff}_\text{loc}(M)$ action on the space $\mathcal{SD}=\{[g]:W_-=0\}$ of self-dual conformal
metric structures. This was essentially done in the recent work \cite[III.A]{DFK}:
By choosing a convenient ansatz the authors of that work encoded all self-dual structures via a
$3\times3$ PDE system $\mathcal{SDE}$ of the second order
(this works for the neutral signature; in the Riemannian case use doubly biorthogonal
coordinates to get self-duality as a $5\times5$ second-order PDE system \cite[III.C]{DFK}
that can be investigated in a similar manner as the $3\times3$ system).

In this way almost all gauge freedom was eliminated, yet a part of symmetry remained shuffling the structures.
This pseudogroup $\mathcal{G}$ is parametrized by 5 functions of 2 arguments (and so is considerably smaller
than $G$). We fix this freedom by computing the differential invariants of $\mathcal{G}$-action on $\mathcal{SDE}$
and passing to the quotient equation.

The differential invariants are considered in rational-polynomial form, as in \cite{KL2}.
This allows to describe the algebra of invariants in Lie-Tresse approach, and also
using the principle of $n$-invariants of \cite{ALV}. We count differential invariants
in both approaches and organize the obtained numbers in the Hilbert polynomial and the Poincar\'e function.

 \section{Scalar invariants of self-dual structures}\label{S2}

The first approach to compute the quotient of the self-duality equation by the local diffeomorphisms
pseudogroup $G$ action is via differential invariants of self-dual structures $\mathcal{SD}$.
The signature of the metric $g$ or conformal metric structure $[g]$ is either $(2,2)$ or $(4,0)$.
In this and the following two sections we assume that $g$ is a Riemannian metric on $M$ for convenience.
Consideration of the case $(2,2)$ is analogous.

To distinguish between metrics and conformal structures we will write $\mathcal{SD}_m$
for the former and $\mathcal{SD}_c$ for the latter.
Denote the space of $k$-jets of such structures by $\mathcal{SD}^k_m$ and $\mathcal{SD}^k_c$ respectively.
These clearly form a tower of bundles over $M$ with projections
$\pi_{k,l}:\mathcal{SD}^k_{\rm x}\to \mathcal{SD}^l_{\rm x}$, $\pi_k:\mathcal{SD}^k_{\rm x}\to M$,
where ${\rm x}$ is either $m$ or $c$.

 \subsection{Self-dual metrics: invariants}

Consider the bundle $S^2_+T^*M$ of positively definite quadratic forms on $TM$ and its space of jets
$J^k(S^2_+T^*M)$. The equation $W_-=0$ in 2-jets determines the submanifold $\mathcal{SD}^2_m\subset J^2$,
and its prolongations are $\mathcal{SD}^k_m\subset J^k$ for $k>2$.

Computation of the stabilizer of the action shows that the submanifolds $\mathcal{SD}^k_m$ are regular,
meaning that generic orbits of the $G$-action in $\mathcal{SD}^k_m$ have the same dimension as in $J^k(S^2_+T^*M)$.
This is based on a simple observation that generic self-dual metrics have no symmetry at all.
Thus the differential invariants of the action on $\mathcal{SD}^k_m$ can be obtained from
the differential invariants on the jet space $J^k$ \cite{K1,LY}.

These invariants can be constructed as follows. There are no invariants of order $\leq1$ due to
existence of geodesic coordinates, the first invariants arise in order $2$ and they are derived from the
Riemann curvature tensor (as this is the only invariant of the 2-jet of $g$). Traces of the Ricci tensor
$\op{Tr}(\op{Ric}^i)$, $1\le i\le4$, yield 4 invariants $I_1,\dots,I_4$ that in a Zariski open set of jets of
metrics can be considered horizontally independent, meaning $\hat dI_1\we\dots\we\hat dI_4\neq0$.

To get other invariants of order 2, choose an eigenbasis $e_1\,\dots,e_4$ of the Ricci operator
(in a Zariski open set it is simple), denote the dual coframe by $\{\theta^i\}$ and decompose
$R_g=R^i_{jkl}e_i\ot\theta^j\ot\theta^k\we\theta^l$. These invariants include the previous $I_i$,
and the totality of independent second-order invariants for self-dual metrics is
 $$
\dim\{R_g|W_-=0\}-\dim O(g)=(20-5)-6=9.
 $$
The invariants $R^i_{jkl}$ are however not algebraic, but obtained as algebraic extensions via the characteristic
equation. Then $R^i_{jkl}$ (9 independent components) and $e_i$ generate the algebra of invariants.

Alternatively, compute the basis of Tresse derivatives $\nabla_i=\hat{\p}_{I_i}$ and express
the metric in the dual coframe $\omega^j=\hat dI_j$: $g=G_{ij}\omega^i \omega^j$.
Then the functions $I_i,G_{kl}$ generate the space of invariants by the principle of $n$-invariants \cite{ALV}.

 \begin{rk}
There is a natural almost complex structure $\hat J$ on the twistor space of self-dual $(M,g)$,
i.e.\ on the bundle $\hat M$ over $M$ whose fiber at $a$ consists of the sphere of orthogonal complex
structures on $T_aM$ inducing the given orientation.
The celebrated theorem of Penrose \cite{P,B} states that self-duality is equivalent to integrability of $\hat J$.
Thus local differential invariants of $g$ can be expressed through semi-global invariants of
the foliation of the three-dimensional complex space $\hat M$ by rational curves.
Similarly in the split signature one gets foliation by $\alpha$-surfaces, and the geometry of this
foliation of $\hat M$ yields the invariants on $M$.
 \end{rk}

We explain how to get rid of non-algebraicity in the next subsection.

 \subsection{Self-dual conformal structures: invariants}

Here the invariants of the second order are obtained from the Weyl tensor as the only conformally invariant part of
the Riemann tensor $R_g$. For general conformal structures a description of the scalar invariants was given
recently in \cite{K2}. In our case $W=W_++W_-$ the second component vanishes, and so we have only
5-dimensional space of curvature tensors $\mathcal{W}$, namely Weyl parts of $R_g$ considered as $(3,1)$ tensors.

Let us fix a representative of the conformal structure $g_0\in[g]$ by the requirement $\|W_+\|^2_{g_0}=1$,
this uniquely determines $g_0$ provided that $W_+$ is non-vanishing in a neighborhood
(in the case of neutral signature we have to require $\|W_+\|^2_{g}\neq0$ for some and hence any metric $g\in[g]$
and then we can fix $g_0$ up to $\pm$ by the requirement $\|W_+\|^2_{g_0}=\pm1$).
Use this representative to convert $W_+$ into a $(2,2)$-tensor, considered as a map
$W_+:\La^2T\to\La^2T$, where $T=T_aM$ for a fixed $a\in M$.

Recall \cite{B} that the operator $W=W_++W_-$ is block-diagonal in terms of
the Hodge $*$-decomposition $\La^2T=\La^2_+T\oplus\La^2_-T$. Thus $W_+:\La^2_+T\to\La^2_+T$ is a map of 3-dimensional spaces and it is traceless of norm 1. For the spectrum $\op{Sp}(W_+)=\{\lambda_1,\lambda_2,\lambda_3\}$
this means $\sum\lambda_i=0$, $\max|\lambda_i|=1$. To conclude, we have only one scalar invariant of order 2,
for which we can take $I=\op{Tr}(W_+^2)$.

To obtain more differential invariants we proceed as follows. It is known that Riemannian conformal structure
in 4D is equivalent to a quaternionic structure (split-quaternionic in the split-signature).
In the domain, where $\op{Sp}(W_+|\La^2_+)$ is simple we even get a hyper-Hermitian structure
(on the bundle $TM$ pulled back to $\mathcal{SD}^2_c$, so no integrability conditions for the operators $J_1,J_2,J_3$)
as follows.

Let $\sigma_i\in\La^2_+$ be the eigenbasis of $W_+$ corresponding to eigenvalues $\lambda_i$,
normalized by $\|\sigma_i\|^2_{g_0}=1$ (this still leaves $\pm$ freedom for every $\sigma_i$).
These 2-forms are symplectic (= nondegenerate, since again these are forms on a bundle over $\mathcal{SD}^2_c$)
and $g_0$-orthogonal, so the operators $J_i=g_0^{-1}\sigma_i$ are
anti-commuting complex operators on the space $T$, and they are in quaternionic relations up to the sign.
We can fix one sign by requiring $J_3=J_1J_2$, but still have residual freedom $\mathbb{Z}_2\times\mathbb{Z}_2$.

Now we can fix a canonical (up to above residual symmetry) frame, depending on the 3-jet of $[g]$, as follows:
$e_1=g_0^{-1} \hat dI/\|g_0^{-1} \hat dI\|_{g_0}$, $e_2=J_1e_1$, $e_3=J_2e_1$, $e_4=J_3e_1$. The structure functions of this frame $c^k_{ij}$ (given by $[e_i,e_j]=c^k_{ij}e_k$) together with $I$ constitute the fundamental invariants
of the conformal structure (we can fix, for instance, $I_1=I$, $I_2=c^1_{12}$, $I_3=c^1_{13}$, $I_4=c^1_{14}$
to be the basic invariants), and together with the invariant derivations $\nabla_j=\mathcal{D}_{e_j}$ 
(total derivative along $e_j$) they generate the algebra of scalar differential invariants micro-locally.

The micro-locality comes from non-algebraicity of the invariants. Indeed, since we used eigenvalues and
eigenvectors in the construction, the output depends on an algebraic extension via some additional variables
$y$. Notice though that this involves only 2-jet coordinates, i.e.\ the $y$-variables are in algebraic
relations with the fiber variables of the projection $J^2\to J^1$, and with respect to higher jets everything
is algebraic. Thus we can eliminate the $y$-variables, as well as the residual freedom, and obtain the
algebra of global rational invariants $\mathfrak{A}_l$.

Here $l$ is the order of jet from which only polynomial behavior of the invariants can be assumed \cite{KL2}.
This yields the Lie-Tresse type description of the algebra $\mathfrak{A}_l$.

It is easy to see that the rational expressions occur at most on the level of 3-jets, so the generators
of the rational algebra can be chosen polynomial in the jets of order $>3$.
Thus we conclude:

 \begin{theorem}
The algebra $\mathfrak{A}_3$ of rational-polynomial invariants as well as the field
$\mathfrak{F}$ of rational differential invariants of self-dual conformal metric structures
are both generated by a finite number of (the indicated) differential invariants $I_i$
and invariant derivations $\nabla_j$, and the invariants from this
algebra/field separate generic orbits in $\mathcal{SD}_c^\infty$.
 \end{theorem}

A similar statement also holds true for metric invariants of $\mathcal{SD}_m^\infty$.

 \section{Stabilizers of generic jets}\label{S3}

Our method to compute the number of independent differential invariants of order $k$ follows the approach
of \cite{LY}. We will use the jet-language from the formal theory of PDE, and refer the reader to \cite{KL1}.

Fix a point $a\in M$. Denote by $\mathbb{D}_k$ the Lie group of $k$-jets of diffeomorphisms
preserving the point $a$. This group is obtained from $\mathbb{D}_1=\op{GL}(T)$ by successive
extensions according to the exact 3-sequence
 $$
0\to\Delta_k \longrightarrow \mathbb{D}_k\longrightarrow \mathbb{D}_{k-1}\to\{e\},
 $$
where $\Delta_k=\{[\vp]_x^k:[\vp]_x^{k-1}=[\op{id}]_x^{k-1}\}\simeq S^kT^*\otimes T$ is Abelian ($k>1$).

Denote by $\op{St}_k\subset\mathbb{D}_{k+1}$ the stabilizer of a generic point
$a_k\in \mathcal{SD}_{\textrm x}^k$, and by $\op{St}^0_k$ its connected component of unity.

 \subsection{Self-dual metrics: stabilizers}

We refer to \cite{LY} for computations of stabilizers and note that even though
the computation there is done for generic metrics, it applies to self-dual metrics as well.
Thus in the metric case the stabilizers are the following:
$\op{St}_0=\op{St}_1=O(g)$, and $\op{St}^0_k=0$ for $k\ge2$.

Consequently the action of the pseudogroup $G$ on jets of order $k\ge2$ is almost free,
meaning that $\mathbb{D}_{k+1}$ has a discrete stabilizer on $\mathcal{SD}^k_m|_a$.

 \subsection{Self-dual conformal structures: stabilizers}

The stabilizers for general conformal structures were computed in \cite{K2}.
In the self-dual case there is a deviation from the general result. Denote by
$\mathcal{C}_M=S^2_+T^*M/\R_+$ the bundle of conformal metric structures.

 \begin{lem} {\rm\bf(\cite{K2})}
The following is a natural isomorphism:
$$T_{[g]}(\mathcal{C}_M\!)=\op{End}^\text{\rm sym}_0(T)=\{A:T\to T\,|\, g(Au,v)=g(u,Av),\op{Tr}(A)=0\}.$$
 \end{lem}

Denote $V_M=T_{[g]}(\mathcal{C}_M\!)$. The differential group $\mathbb{D}_{k+1}$ acts on
$\mathcal{SD}^k_c$, in particular $\Delta_{k+1}$ acts on it.
The next statement is obtained by a direct computation of the symbol of Lie derivative.

 \begin{lem}
The tangent to the orbit $\Delta_{k+1}(a_k)$ is the image $\op{Im}(\zeta_k)\subset T\mathcal{SD}^k_c$ of the map
$\zeta_k$ that is equal to the following composition
 $$
S^{k+1}T^*\ot T\stackrel{\delta}\longrightarrow S^kT^*\otimes(T^*\ot T)
\stackrel{\1\otimes\Pi}\longrightarrow S^kT^*\ot V_M.
 $$
Here $\delta$ is the Spencer operator and $\Pi:T^*\ot T\to V_M\subset T^*\ot T$
is the projection given by
 $$
\langle p,\Pi(B)u\rangle=\tfrac12\langle p,Bu\rangle
+\tfrac12\langle u_\flat,Bp^\sharp\rangle-\tfrac1n\op{Tr}(B)\langle p,u\rangle,
 $$
where $u\in T,p\in T^*,B\in T^*\ot T$ are arbitrary,
$\langle\cdot,\cdot\rangle$ denotes the pairing between $T^*$ and $T$,
and $u_\flat=g(u,\cdot)$, $p^\sharp=g^{-1}(p,\cdot)$ for some representative $g\in[g]$,
on which the right-hand side does not depend.
 \end{lem}

Recall that $i$-th prolongation of a Lie algebra $\mathfrak{h}\subset\op{End}(T)$
is defined by the formula $\mathfrak{h}^{(i)}=S^{i+1}T^*\ot T\cap S^iT^*\ot\mathfrak{h}$.
As is well-known, for the conformal algebra of $[g]$ it holds:
$\mathfrak{co}(g)^{(1)}=T^*$ and $\mathfrak{co}(g)^{(i)}=0$, $i>1$.

 \begin{lem}\label{Lzeta}
We have $\op{Ker}(\zeta_k)=0$ for $k>1$, and therefore
the projectors $\rho_{k+1,k}:\mathbb{D}_{k+1}\to\mathbb{D}_k$ induce the injective homomorphisms
$\op{St}_k\to\op{St}_{k-1}$ and $\op{St}^0_k\to\op{St}^0_{k-1}$ for $k>1$.
 \end{lem}

 \begin{proof}
If $\zeta_k(\Psi)=0$, then $\delta(\Psi)\in S^kT^*\ot\mathfrak{co}(g)$, where
$\mathfrak{co}(g)\subset\op{End}(T)$ is the conformal algebra. This means that
$\Psi\in\mathfrak{co}(g)^{(k+1)}=0$, if $k>1$. Thus we conclude injectivity of $\zeta_k$:
$\Delta_{k+1}\cap\op{St}_k=\{e\}$, whence the second claim.
 \end{proof}

The stabilizers of low order (for any $n\ge3$) are the following. For any $a_0\in \mathcal{C}_M$
its stabilizer is $\op{St}_0=CO(g)=(\op{Sp}(1)\times_{\mathbb{Z}_2}\op{Sp}(1))\times\R_+$.

Next, the stabilizer $\op{St}_1\subset\mathbb{D}_2$ of
$a_1\in J^1(\mathcal{C}_M\!)$ is the extension
(by derivations) of $\op{St}_0$ by $\mathfrak{co}(g)^{(1)}=T^*\stackrel{\iota}\hookrightarrow\Delta_2$,
where $\iota:T^*\to S^2T^*\ot T$ is given by
 $$
\iota(p)(u,v)=\langle p,u\rangle v+\langle p,v\rangle u
-\langle u_\flat,v\rangle p^\sharp,
 $$
for $p\in T^*$, $u,v\in T$.
In other words, we have $\op{St}_1=CO(g)\ltimes T$.

Since for $G$-action on $\mathcal{SD}^2_c$ there is precisely 1 scalar differential invariant,
we get $\dim\op{St}_2=(16+40+80)-(9+36+85-1)=7$.
This can be also seen as follows. Since $\op{St}^0_2\subset\op{St}_1$ preserves the hyper-Hermitian
structure determined by generic 2-jet $a_2\in\mathcal{SD}_c^2$ (see Section \ref{S2})
the $\R_+$ factor and one of the $\op{Sp}(1)$ copies in $\op{St}_0$ disappears from the stabilizer of 2-jet,
and we get $\op{St}^0_2\simeq \op{Sp}(1)\ltimes T$.

 \begin{lem}\label{St}
For $k\ge3$ we have: $\op{St}^0_k=\{e\}$.
 \end{lem}

 \begin{proof}
In Section \ref{S2} we constructed a canonical frame $e_1,\dots,e_4$ on $T$ depending on
(generic) jet $a_3$. In other words, we constructed a frame on the bundle $\pi_3^*TM$
over a Zariski open set in $\mathcal{SD}_c^3$.

The elements from $\op{St}^0_3$ shall preserve this frame, and so the last component $\op{Sp}(1)$
from $\op{St}_0$ is reduced. But also the elements from $\op{St}^0_3$ shall preserve the 1-jet of the
hyper-Hermitian structure and the invariant $I$ determined by 2-jets, whence also the factor $T$ is reduced,
and $\op{St}^0_3$ is trivial (we take the connected component because of the undetermined signs $\pm$
in the normalizations). Hence the stabilizers $\op{St}^0_k$ for $k\ge 3$ are trivial as well.
 \end{proof}

 \section{Hilbert polynomial and Poincar\'e function for $\mathcal{SD}$}\label{S4}

Now we can compute the number of independent differential invariants. Since $G$ acts transitively on $M$
the codimension of the orbit of $G$ in $\mathcal{SD}^k_{\rm x}$ is equal to the codimension of the
orbit of $\mathbb{D}_{k+1}$ in $\mathcal{SD}^k_{\rm x}|_a$ (where $a\in M$ is a fixed point
and ${\rm x}$ is either $m$ or $c$).
Denoting the orbit through a generic $k$-jet $a_k$ by $\mathcal{O}_k\subset\mathcal{SD}^k_{\rm x}|_a$ we have:
 $$
\dim(\mathcal{O}_k)=\dim\mathbb{D}_{k+1}-\dim\op{St}_k.
 $$
Notice that
 $$
\op{codim}(\mathcal{O}_k)=\dim \mathcal{SD}^k_{\rm x}|_a-\dim(\mathcal{O}_k)=\op{trdeg}\mathfrak{F}_k
 $$
is the number of (functionally independent) scalar differential invariants of order $k$
(here $\op{trdeg}\mathfrak{F}_k$ is the transcendence degree of the field of rational differential invariants on
$\mathcal{SD}^k_{\rm x}$).

The Hilbert function is the number of ``pure order" $k$ differential invariants
$H(k)=\op{trdeg}\mathfrak{F}_k-\op{trdeg}\mathfrak{F}_{k-1}$.
It is known to be a polynomial for large $k$,
so we will refer to it as the Hilbert polynomial.

The Poincar\'e function is the generating function for the Hilbert polynomial, defined by
$P(z)=\sum_{k=0}^\infty H(k)z^k$. This is a rational function with the only pole $z=1$ of
order equal to the minimal number of invariant derivations in the Lie-Tresse generating set \cite{KL2}.

 \subsection{Counting differential invariants}
The results of Section \ref{S3} allow to compute the Hilbert polynomial and the Poincar\'e function.

 \begin{theorem}
The Hilbert polynomial for $G$-action on $\mathcal{SD}_m$ is
 $$
H_m(k)=\left\{\begin{array}{ll}
0 & \text{ for }k<2,\\
9 & \text{ for }k=2,\\
\tfrac16(k-1)(k^2+25k+36) & \text{ for }k>2.
\end{array}\right.
 $$
The corresponding Poincar\'e function is equal to
 $$
P_m(z)=\frac{z^2(9+4z-30z^2+24z^3-6z^4)}{(1-z)^4}.
 $$
 \end{theorem}

Notice that $H_m(k)\sim\tfrac1{3!}\,k^3$, meaning that the moduli of self-dual metric structures
are parametrized by $1$ function of $4$ arguments. This function is the unavoidable rescaling factor.

 \begin{proof}
As for the general metrics, there are no invariants of order $<2$. Since $\op{St}^0_2=0$, we have:
 $$
H_m(2)=\dim\mathcal{SD}^2_m|_a-\dim\mathbb{D}_3=(10+40+95)-(16+40+80)=9.
 $$
Alternatively, the only  invariant of the 2-jet of a metric is the Riemann curvature tensor.
Since $W_-=0$, it has $20-5=15$ components and is acted upon effectively by the group $O(g)$ of dimension 6;
hence the codimension of a generic orbit is $15-6=9$.

Starting from 2-jet we impose the self-duality constraint that, as discussed in the introduction,
consist of 5 equations and is a determined system (mod gauge). In particular, there are no
differential syzygies between these 5 equations, so that in ``pure" order $k\ge2$ the number of independent
equations is $5\cdot\binom{k+1}3$. Thus the symbol of the self-duality metric equation $W_-=0$ on $g$, given by
 $$\g_k=\op{Ker}(d\pi_{k,k-1}:T\mathcal{SD}_m^k\to T\mathcal{SD}_m^{k-1})$$
has dimension $\dim(S^kT^*\ot S^2T^*)-\#[\text{independent equations}]$. 

Since the pseudogroup $G$ acts almost freely on jets of order $k\ge2$ (freely from some order $k$), we have:
 $$
H_m(k)=\dim\g_k-\dim\Delta_{k+1}=10\cdot\binom{k+3}3-5\cdot\binom{k+1}3-4\cdot\binom{k+4}3
 $$
whence the claim for the Hilbert polynomial.
The formula for the Poincar\'e function follows.
 \end{proof}

 \begin{theorem}\label{HandP}
The Hilbert polynomial for $G$-action on $\mathcal{SD}_c$ is
 $$
H_c(k)=\left\{\begin{array}{ll}
0 & \text{ for }k<2,\\
1 & \text{ for }k=2,\\
13 & \text{ for }k=3,\\
3k^2-7 & \text{ for }k>3.
\end{array}\right.
 $$
The corresponding Poincar\'e function is equal to
 $$
P_c(z)=\frac{z^2(1+10z+5z^2-17z^3+7z^4)}{(1-z)^3}.
 $$
 \end{theorem}

Notice that $H_c(k)\sim6\cdot\tfrac1{2!}\,k^2$, meaning that the moduli of self-dual conformal metric structures
are parametrized by $6$ function of $3$ arguments. This confirms the count in \cite{G,DFK}.

 \begin{proof}
As for the general metrics, there are no invariants of order $<2$.
We already counted $H_c(2)=1$. Since $\op{St}^0_3=0$, we have:
 \begin{multline*}
H_c(3)=\dim\mathcal{SD}^3_m|_a-\dim\mathbb{D}_4-H_c(2)\\
=(9+36+85+160)-(16+40+80+140)-1=13.
 \end{multline*}

Starting from 2-jet we impose the self-duality constraint, and we computed in the previous proof that this
yields $5\cdot\binom{k+1}3$ independent equations of ``pure" order $k\ge2$.
Thus the symbol of the self-duality conformal equation $W_-=0$ on $[g]$, given by
 $$
\g_k=\op{Ker}(d\pi_{k,k-1}:T\mathcal{SD}_c^k\to T\mathcal{SD}_c^{k-1}),
 $$
has dimension$=\dim(S^kT^*\ot(S^2T^*/\R_+))-\#[\text{independent equations}]$.

Since the pseudogroup $G$ acts almost freely on jets of order $k\ge3$ (freely from some order $k$), we have:
 $$
H_c(k)=\dim\g_k-\dim\Delta_{k+1}=9\cdot\binom{k+3}3-5\cdot\binom{k+1}3-4\cdot\binom{k+4}3
 $$
whence the claim for the Hilbert polynomial.
The formula for the Poincar\'e function follows.
 \end{proof}

 \subsection{The quotient equation}
Let $I_1,\dots,I_4$ be the basic differential invariants of self-dual conformal structures.
For generic such structures $c$ these invariant evaluated on $c$ are independent.
Thus we can fix the gauge by requiring $I_i=x_i$, $i=1,\dots,4$, to be the local coordinates on $M$.
This adds 4 differential equations to 5 equations of self-duality on 9 components of $c$.
Consequently, denoting
 $$
\Sigma_\infty=\{\theta\in\mathcal{SD}^\infty_c:\hat dI_1\we\hat dI_2\we\hat dI_3\we\hat dI_4
\text{ is not defined at }\theta\text{ or vanishes}\},
 $$
the moduli space $(\mathcal{SD}_c^\infty\setminus\Sigma_\infty)/G$ is given as $9\times9$ PDE system
 $$
W_-=0, I_1=x_1, \dots, I_4=x_4.
 $$

 \section{The self-duality equation}\label{S5}
 In the second approach we use a $3 \times 3$ PDE system from \cite{DFK} which encodes all self-dual conformal structures. It was shown in loc.cit.\ that any anti-self-dual conformal structure in neutral signature $(2,2)$ locally takes the form $[g]$ where
 \begin{equation}
 \label{ASDmetric}
 g=dtdx+dzdy+p\,dt^2+2q\,dtdz+r\,dz^2.
 \end{equation}
 Here $p,q,r$ are functions of $(t,x,y,z)$ which satisfy the following three second-order PDEs:
  \begin{equation}
 \begin{array}{c}
 p_{xx}+2q_{xy}+r_{yy}=0,\\
 \ \\
 m_x+n_y=0,\\
 \ \\
 m_z-qm_x-rm_y+(q_x+r_y)m=n_t-pn_x-qn_y+(p_x+q_y)n,
 \end{array}
 \label{SDE}
  \end{equation}
 where
$$
m:= p_z-q_t+pq_x-qp_x+qq_y-rp_y, ~~~
n:=q_z-r_t+qr_y-rq_y+pr_x-qq_x.
 $$

Conversely, any such conformal structure is anti-self-dual. Therefore we can, instead of looking at arbitrary self-dual conformal structures, look at conformal structures $[g]$ where $g$ is a metric of the Pleba\'nski-Robinson form (\ref{ASDmetric}) satisfying (\ref{SDE}). So from now on we restrict to self-dual conformal structures in 
the neutral signature $(2,2)$.

\begin{rk}
These equations are admittedly describing anti-self-dual metrics ($*W=-W$) instead of self-dual metrics ($*W=W$). However, in order to define the Hodge operator, one must specify an orientation. Change of orientation interchanges the equations, so from a local viewpoint self-dual and anti-self-dual structures are the same.
\end{rk}

Conformal structures of the form (\ref{ASDmetric}) are parametrized by sections of the bundle
$\pi\colon\mathcal{C}_M^{\text{PR}} = M\times \mathbb R^3(p,q,r) \to M$, where $M=\mathbb R^4(t,x,y,z)$.
Self-dual conformal structures must, in addition, satisfy system (\ref{SDE}),
so they are described by a second-order PDE
 \[
\mathcal{SDE}_2 =\{\theta=[(p,q,r)]_x^2 : x \in M, \theta \text{ satisfies (\ref{SDE})}\} \subset
J^2(\mathcal{C}_M^{\text{PR}}).
 \]
We let $\mathcal{SDE}_k \subset J^k=J^k(\mathcal{C}_M^{\text{PR}})$ denote the prolonged equation.
From now on we will omit specification of the bundle over which the jet spaces are constructed, because
it will always be $\mathcal{C}_M^{\text{PR}}$ in what follows.

The prolonged equation $\mathcal{SDE}_k$ is given by $3 \binom{k+2}{4}$ equations in
$J^{k}$ since the system (\ref{SDE}) is determined.
By subtracting this from the jet space dimension $\dim J^k = 4+3 \binom{k+4}{4}$, we find

\begin{align*}
	\dim \mathcal{SDE}_k = 4+3 \binom{k+4}{4}-3 \binom{k+2}{4} = k^3+\frac{9}{2} k^2+\frac{13}{2} k+7.
\end{align*}

\section{Symmetries of $\mathcal{SDE}$}
Self-dual conformal structures locally correspond to sections of $\mathcal{C}_M^\text{PR}$ that are solutions of $\mathcal{SDE}$. This correspondence is not 1-1 as there is some residual freedom left: two solutions of $\mathcal{SDE}$ can still be equivalent up to diffeomorphisms. The goal is to remove this freedom by factoring by diffeomorphisms that preserve the shape of the conformal structure $[g]$ where $g$ is in Pleba\'nski-Robinson form (\ref{ASDmetric}).

These transformations form the symmetry pseudogroup $\mathcal{G}$ of the equation $\mathcal{SDE}$.
We will study its Lie algebra $\mathfrak{g}$. By the Lie-B\"acklund theorem \cite{KLV} for our equation
all symmetries are (prolongations of) point transformations. It turns out that the Lie algebra of
symmetries is the same as the Lie algebra of vector fields preserving the shape of $[g]$.

\subsection{Symmetries of $\mathcal{SDE}$}
A vector field $X$ on $J^0$ is a symmetry of $\mathcal{SDE}$ if the prolonged vector field $X^{(2)}$ is tangent to $\mathcal{SDE}_2 \subset J^2$, i.e. if $X^{(2)} (F_i) = \lambda_i^j F_j$, where $F_1=0,F_2=0,F_3=0$ are the three equations (\ref{SDE}). This gives an overdetermined system of PDEs that can be solved by the standard technique,
and we obtain the following result:


 \begin{theorem}\label{Thm:Symm}
The Lie algebra $\sym$ of symmetries of $\mathcal{SDE}$ is generated by the following five classes of vector fields $X_1(a)$, $X_2(b)$, $X_3(c)$, $X_4(d)$, $X_5(e)$, each of which depends on a function of $(t,z)$:
{\footnotesize
\begin{align*}
&a \partial_t-x a_t \partial_x-x a_z \partial_y+ (x a_{tt}-2p a_t) \partial_p+(x a_{tz}-q a_t-p a_z) \partial_q+(x a_{zz}-2q a_z) \partial_r, \\
&b \partial_z-y b_t \partial_x-y b_z \partial_y+ (y b_{tt}-2q b_t) \partial_p+(y b_{tz}-q b_z-r b_t) \partial_q+(y b_{zz}-2r b_z) \partial_r, \\
&c x \partial_x+c y \partial_y+(cp-xc_t) \partial_p+(c q-\tfrac{1}{2} x c_z- \tfrac{1}{2} y c_t) \partial_q+(c r-y c_z) \partial_r,\\
&d \partial_x - d_t \partial_p - \tfrac{1}{2} d_z \partial_q, \\
&e \partial_y - \tfrac{1}{2} e_t \partial_q - e_z \partial_r.
\end{align*}}
 \end{theorem}

The following table shows the commutation relations.
\begin{table}[ht]
\makebox[\textwidth][c]{
\def\arraystretch{1.1}
\tiny
\begin{tabular}{|c||c|c|c|c|c|}
\hline
$[,]$ & $X_1(g)$ & $X_2(g)$ & $X_3(g)$ & $X_4(g)$ & $X_5(g)$ \\\hline \hline
$X_1(f)$ & $X_1(f g_t-f_t g)$ & $X_2(f g_t)-X_1(f_z g)$ & $X_3(f g_t)$ &$X_4((f g)_t)+X_5(f_z g)$ & $X_5(f g_t)$ \\\hline
$X_2(f)$ & $*$ & $X_2(f g_z-f_z g)$&$ X_3(f g_z)$ & $X_4(f g_z)$ & $X_4(f_t g)+X_5((fg)_z)$ \\\hline
$X_3(f)$ &$*$&$*$&$0$ & $-X_4(fg)$ & $-X_5 (fg)$\\\hline
$X_4(f)$ &$*$&$*$&$*$& $0$ & $0$\\\hline
$X_5(f)$ &$*$&$*$&$*$&$*$& $0$\\\hline
\end{tabular}
}
\end{table}

Notice that the Lie algebra is bi-graded $\sym= \oplus \mathfrak g_{i,j}$, meaning that  $[\mathfrak g_{i_1,j_1},\mathfrak g_{i_2,j_2}] \subset \mathfrak g_{i_1+i_2,j_1+j_2}$ with nontrivial graded pieces \[\mathfrak g_{0,0} = \langle X_1,X_2 \rangle, \qquad \mathfrak g_{0,1} = \langle X_3 \rangle, \qquad \mathfrak g_{1,\infty} = \langle X_4,X_5 \rangle.\]


\subsection{Shape-preserving transformations}
We say that a transformation $\varphi \in \text{Diff}_\text{loc}(M)$ preserves the PR-shape if for every $[g] \in \Gamma(\mathcal{C}_M^{\text{PR}})$ we have $[\varphi_* g] \in \Gamma(\mathcal{C}_M^{\text{PR}})$. A vector field $X$ on $\mathbb R^4$ preserves the PR-shape if its flow does so.

 \begin{theorem}\label{Thm:PRshape}
The Lie algebra of vector fields preserving the PR-shape is generated by the five classes of vector fields
 \begin{gather*}
a \partial_t-x a_t \partial_x-x a_z \partial_y,\ \
b \partial_z-y b_t \partial_x-y b_z \partial_y,\ \
c x \partial_x+c y \partial_y,\ \ d \partial_x,\ \ e \partial_y.
\end{gather*}
 %
where $a,b,c,d,e$ are arbitrary functions of $(t,z)$.
 \end{theorem}

\begin{proof}
In order to find the Lie algebra of vector fields preserving the shape of $[g]$, we let $X = f_1 \partial_t+f_2 \partial_x+f_3 \partial_y+f_4 \partial_z$ be a general vector field and take the Lie derivative $L_X g$. The vector field preserves the PR-shape of $[g]$ if
 \[
L_X g = \epsilon \cdot (dt dx+dz dy) + \tilde p\,dt^2+2 \tilde q\,dt dz+\tilde r\,dz^2
 \]
for some functions $\epsilon,\tilde p,\tilde q, \tilde r$. This gives an overdetermined system of 6 PDEs on 4 unknowns with the solutions parametrized by 5 functions of 2 variables as indicated.
\end{proof}

\subsection{Unique lift to $J^0$}
The conformal metric (\ref{ASDmetric}) can also be considered as a horizontal (degenerate) symmetric tensor $c_{\text{PR}}$ on $\mathcal{C}_M^{\text{PR}}$. Namely, $c_{\text{PR}}\in\Gamma(\pi^*S^2T^*M/\mathbb{R}_+)$ is given at the point $(t,x,y,z,p,q,r)\in\mathcal{C}_M^\text{PR}$ via its representative $g$ by formula (\ref{ASDmetric}).
The algebra of vector fields $X$ preserving the shape of $[g]$ is naturally lifted to $\mathcal{C}_M^{\text{PR}}$ by the requirement $L_{\hat X} c_\text{PR} =0$. This requirement algebraically restores the vertical components
of the vector fields $X_1,\dots,X_5$ from Theorem \ref{Thm:PRshape} yielding the symmetry fields from
Theorem \ref{Thm:Symm}. We conclude:

 \begin{theorem}
The Lie algebra of transformations preserving the PR-shape coincides with
the Lie algebra $\mathfrak{g}$ of point symmetries of $\mathcal{SDE}$.
 \end{theorem}
Thus the conformal structure $c_{\text{PR}}$ uniquely restores $\sym=\text{sym}(\mathcal{SDE})$.

\subsection{Conformal tensors invariant under $\sym$}\label{ss:confinv}
The goal of this subsection is to show that the simplest conformally invariant tensor with respect to $\sym$ is $c_{PR}$, so that the conformal structure (of PR-shape) is in turn uniquely determined by $\sym$.

We aim to describe the horizontal conformal tensors on $\mathcal{C}_M^\text{PR}$ that are invariant with respect to $\sym$. Since $\sym$ acts transitively on $\mathcal{C}_M^\text{PR}$, we consider the stabilizer $\text{St}_0 \subset \sym$ of the point given by $(t,x,y,z,p,q,r)=(0,0,0,0,0,0)$ in $\mathcal{C}_M^\text{PR}$.
Denote by $\text{St}^k_0$ the subalgebra of $\g$ consisting of fields vanishing at $0$ to order $k$,
so that $\text{St}_0=\text{St}^1_0$.

It is easy to see from formulae of Theorem \ref{Thm:Symm} that the space $\text{St}^1_0/\text{St}^2_0$
is 18-dimensional, and 12 of the generators are vertical (belong to $\langle\p_p,\p_q,\p_r\rangle$).
The complimentary linear fields have the horizontal parts
 \begin{align*}
Y_1=t \partial_t-x \partial_x, \qquad & Y_2= z \partial_t -x \partial_y, \qquad Y_3= t \partial_z-y\partial_x, \\
Y_4 = z \partial_z-y \partial_y, \qquad & Y_5=x \partial_x  + y \partial_y, \qquad Y_6= z \partial_x-t \partial_y.
 \end{align*}
They form a 6-dimensional Lie algebra $\mathfrak{h}$ acting on the horizontal space
$\mathbb{T}=T_0M=T_0\mathcal{C}_M^\text{PR}/\op{Ker}(d\pi)$. This Lie algebra is a semi-direct product of the reductive part
$\mathfrak{h}_0=\langle Y_1,Y_2,Y_3,Y_4,Y_5\rangle$ and the nilpotent piece $\mathfrak{r}=\langle Y_6\rangle$
(the nilradical is 2-dimensional). The reductive piece splits in turn
$\mathfrak{h}_0=\mathfrak{sl}_2\oplus\mathfrak{a}$, where the semi-simple part is
$\mathfrak{sl}_2=\langle Y_1-Y_4,Y_2,Y_3\rangle$ and the Abelian part is
$\mathfrak{a}=\langle Y_1+Y_4,Y_5\rangle$.

It is easy to see that the space $\mathbb{T}$ is $\mathfrak{h}_0$-reducible.
In fact, with respect to $\mathfrak{h}_0$ it is decomposable $\mathbb{T}=\Pi_1 \oplus \Pi_2 =\langle \partial_t, \partial_z \rangle \oplus \langle \partial_x,\partial_y \rangle$, and $\Pi_1,\Pi_2$ are the standard
$\mathfrak{sl}_2$-representations (denoted by $\Pi$ in what follows).
However $\mathfrak{r}$ maps $\Pi_1$ to $\Pi_2$ and $\Pi_2$ to 0.
This $\Pi_2\subset\mathbb{T}$ is an $\mathfrak{h}$-invariant subspace, but it does not have an
$\mathfrak{h}$-invariant complement.

Moreover, $\Pi_2$ is the only proper $\mathfrak{h}$-invariant subspace, so there are no conformally invariant
vectors (invariant 1-space) and covectors (invariant 3-space). We sumarize this as follows.
 \begin{lem}
There are no horizontal 1-tensors on $\mathcal{C}_M^\text{PR}$ that are
conformally invariant with respect to $\mathfrak{g}$.
 \end{lem}

Now, let's consider conformally invariant horizontal 2-tensors.
Since $c_\text{PR}$ is $\mathfrak{g}$-invariant, we can lower the indices and consider $(0,2)$-tensors.
We have the splitting $\mathbb{T}^*\otimes\mathbb{T}^* = \La^2\mathbb{T}^* \oplus S^2\mathbb{T}^*$.

The symmetric part further splits $S^2(\Pi_1^*\oplus \Pi_2^*) = S^2 \Pi_1^*\oplus (\Pi_1^* \otimes \Pi_2^*) \oplus S^2 \Pi_2^*$. As an $\mathfrak{sl}_2$-representation, this is equal to
$3\cdot S^2\Pi\oplus\Lambda^2\Pi=3\cdot\mathfrak{ad}\oplus\mathbf{1}$,
and the only one trivial piece $\mathbf{1}\subset \Pi_1^*\otimes\Pi_2^*$
(which is also $\mathfrak{h}$-invariant) is spanned by $c_\text{PR}$.
Here $\Pi_1^*=\langle dt,dz \rangle$ and $\Pi_2^*=\langle dx,dy \rangle$.
Thus there are no $\mathfrak{g}$-invariant symmetric conformal 2-tensors except $c_\text{PR}$.

The skew-symmetric part further splits $\La^2(\Pi_1^*\oplus\Pi_2^*)=\La^2\Pi_1^*\oplus(\Pi_1^*\otimes\Pi_2^*)\oplus \La^2 \Pi_2^*$, and as an $\mathfrak{sl}_2$-representation, this is equal to
$S^2\Pi\oplus 3\cdot \La^2\Pi=\mathfrak{ad}\oplus3\cdot\mathbf{1}$. Thus there are three $\mathfrak{sl}_2$-trivial
pieces, and they are $\mathfrak{h}_0$-invariant. However only one of them is $\mathfrak{r}$-invariant,
namely $\La^2\Pi_1^*$ that is spanned by $dz \wedge dt$. Thus we have proved the following statement.

\begin{theorem}
The only conformally invariant symmetric 2-tensor is $c_\text{PR}$. The only conformally invariant skew-symmetric 2-tensor is $dz \wedge dt$.
\end{theorem}
Since $dz \wedge dt$ is degenerate and does not define a convenient geometry, $c_\text{PR}$ is the simplest $\sym$-invariant conformal tensor.

\subsection{Algebraicity of $\sym$}
We say that the Lie algebra $\sym$ is algebraic if its sheafification is equal to the Lie algebra sheaf of
some algebraic pseudo-group $\mathcal G$ (see definition of an algebraic pseudo-group in \cite{KL2}). Algebraicity of $\sym$ is important because it guarantees, through the global Lie-Tresse theorem \cite{KL2}, existence of rational differential invariants separating generic orbits (by \cite{R} this yields rational quotient of the action on every finite jet-level).

Let $\mathbb{D}_k \subset J^k_{(\theta,\theta)}(\mathcal{C}_M^\text{PR}, \mathcal{C}_M^\text{PR})$ denote the differential group of order $k$ at $\theta \in \mathcal{C}_M^\text{PR}$. The stabilizer $\mathcal G_\theta \subset \mathcal G$ of $\theta$ can be viewed as a collection of subbundles $\mathcal G_\theta^k \subset \mathbb{D}_k$. The transitive Lie pseudo-group $\mathcal G$ is algebraic if $\mathcal G_\theta^k$ is an algebraic subgroup of $\mathbb{D}_k$ for every $k$. This is independent of the choice of $\theta$ since $\mathcal G$ is transitive, implying that subgroups $\mathcal G_\theta^k \subset \mathbb{D}_k$ are conjugate for different points $\theta \in \mathcal{C}_M^\text{PR}$.

When determining whether $\sym$ is algebraic, there are essentially two approaches. One is to try to see it from the stabilizer $\sym_\theta$ alone, and the other is to integrate $\sym$ in order to investigate the pseudo-group $\mathcal G_\theta$. It turns out that the latter is more efficient in our case.

Consider the following pseudo-group $\mathcal G$ given via its action on $\mathcal{C}_M^\text{PR}$,
where $A,B,C,D,E$ are arbitrary functions of $(z,t)$.
\begin{gather*}
  t \mapsto T= A,\quad
  z \mapsto Z=B\\
  x \mapsto X=x \frac{C}{A_t} - y B_t+D, \quad
  y \mapsto Y=y \frac{C}{B_z}-x A_z+E \\
  p \mapsto P=p \frac{C}{A_t^2}- D_t - x C_t+y B_{tt}-2q B_t+x A_{tt}\\
  q \mapsto Q=q \frac{C}{B_z A_t} - \tfrac{1}{2} (E_t+D_z+x C_z+y C_t)  + y B_{tz}-r B_t+x A_{tz}-p A_z \\
  r \mapsto R=r \frac{C}{B_z^2} - E_z-y C_z+y B_{zz}+x A_{zz}-2q A_z
\end{gather*}
It is easy to check that this is a Lie pseudo-group (one should specify the differential equations
defining $\mathcal G$, and they are $T_x=0,\dots,T_r=0,\dots,X_y+Z_t=0,\dots$).
Moreover it is easy to check that the Lie algebra sheaf of $\mathcal G$ coincides with the
sheafification of $\mathfrak{g}$.

\begin{theorem} \label{Algebraicity}
  The Lie pseudo-group $\mathcal G$ and consequently the Lie algebra $\sym$ are algebraic.
\end{theorem}

 \begin{proof}
The subgroups $\mathcal G_\theta^k$ of $\mathbb{D}_k$ are constructed by repeated differentiation of $T,...,R$ by $t,...,r$ and evaluation at $\theta$. The formulas for the group action make it clear that $\mathcal G_\theta^k$ will always be an algebraic subgroup of $\mathbb{D}_k$ (they provide a rational parametrization of it as a subvariety).
Thus $\mathcal G$ is algebraic. The statement for $\g$ follows.
 \end{proof}

\medskip

Let us briefly explain how to read algebraicity from the Lie algebra $\g$. Consider the Lie subalgebra $\mathfrak{f}\subset\mathfrak{gl}(T_0J^0)$ obtained by linearization of the isotopy algebra at
$0\in J^0=\mathcal{C}_M^\text{PR}$.
As already noticed in \S\ref{ss:confinv}, this is an 18-dimensional subalgebra admitting
the following exact 3-sequence
 $$
0\to\mathfrak{v}\longrightarrow\mathfrak{f}\longrightarrow\mathfrak{h}\to0,
 $$
where $\mathfrak{v}$ is the vertical part and $\mathfrak{h}$ -- the "horizontal" (that is the quotient).
The explicit form of these vector fields come from Theorem~\ref{Thm:Symm}:
 \begin{gather*}
\mathfrak{v}=\langle x\p_p,x\p_q,x\p_r,y\p_p,y\p_q,y\p_r,t\p_p,t\p_q,t\p_r,z\p_p,z\p_q,z\p_r\rangle,\\
\mathfrak{h}=\mathfrak{sl}_2+\mathfrak{a}+\mathfrak{r},\ \text{ where }\quad
\mathfrak{r}=\langle z\p_x-t\p_y\rangle,\\
\mathfrak{sl}_2=\langle z\p_t-x\p_y-p\p_q-2q\p_r, t\p_z-y\p_x-2q\p_p-r\p_q,\\
\hspace{82pt} t\p_t-z\p_z-x\p_x+y\p_y-2p\p_p+2r\p_r \rangle,\\
\mathfrak{a}=\langle t\p_t+z\p_z-p\p_p-q\p_q-r\p_r,x\p_x+y\p_y+p\p_p+q\p_q+r\p_r\rangle.
 \end{gather*}
By \cite{CT} the subalgebra $[\mathfrak{f},\mathfrak{f}]\subset\mathfrak{gl}(T_0J^0)$ is algebraic.
Since $\mathfrak{f}$ is obtained from
$[\mathfrak{f},\mathfrak{f}]=\mathfrak{v}+\mathfrak{sl}_2+\mathfrak{r}$ by extension by derivations $\mathfrak{a}$,
and the semi-simple elements in the latter have no irrational ratio of spectral values,
we conclude that $\mathfrak{f}\subset\mathfrak{gl}(T_0J^0)$ is an algebraic Lie algebra \cite{Ch}.
The claim about algebraicity of $\g$ follows by prolongations.

\section{Hilbert polynomial and Poincar\'e function for $\mathcal{SDE}$}\label{S6}

Even though $\sym$ is just a PR-shape preserving Lie algebra, its prolongation to the space of 2-jets
preserves $\mathcal{SDE}$ (this is an unexpected remarkable fact), and we consider the orbits of $\sym$ on this equation.


\subsection{Dimension of generic orbits}
We can compute the dimension of a generic orbit in $\mathcal{SDE}_k$ or $J^{k}$ by computing the rank of the system of prolonged symmetry vector fields $X^{(k)}$ at a point in general position.

By prolonging the generators $X_1,...,X_5$ and with the help of Maple we observe that the Lie algebra $\sym$ acts transitively on $J^1$. The dimension of a generic orbit on the Lie algebra acting on $J^2$ is 44, but the equation $\mathcal{SDE}_2 \subset J^2$ contains no generic orbits, and if we restrict to $\mathcal{SDE}_2$ a generic orbit of $\sym$ is of dimension 42. For higher jet-orders $k>2$, the dimension of a generic orbit is the same on $\mathcal{SDE}_k$ as on $J^{k}$.

We are going to compute $\dim\mathcal O_k$ for $k \geq 3$ as follows. Since $\sym$ contains the translations $\partial_t, \partial_z$, all its orbits pass through the subset $S_k \subset J^k$ given by $t=0,z=0$. On $S_k$ we can make the Taylor expansion of parametrizing functions $a,b,c,d,e$ around $(t,z)=(0,0)$.

We use $X_5(e)$ to show the idea. By varying the coefficients of the Taylor series $e(t,z)=e(0,0)+e_t(0,0)t+e_z(0,0)z+\cdots $ we see that the vector fields $X_5 (m,n) = z^m t^n \partial_y- \tfrac{n}{2} z^m t^{n-1} \partial_q - mz^{m-1} t^n \partial_r$  are contained in the symmetry algebra, with the convention that $t^{-1}=z^{-1}=0$, and any vector field of the form $X_5(e)$ is tangent to a vector field in $\langle X_5 (m,n)\rangle$. The prolongation of a vector field takes the form
 \begin{equation}\label{eq:prolongation}
X^{(k)} = \sum_i a_i\D_i^{(k+1)}+
\sum_{|\sigma| \leq k} (\D_\sigma (\phi_p) \partial_{p_\sigma}+\D_\sigma (\phi_q) \partial_{q_\sigma}+\D_\sigma (\phi_r) \partial_{r_\sigma})
 \end{equation}
where $\D_\sigma$ is the iterated total derivative, $\D_i^{(k+1)}$ the truncated total
derivative (``restriction" to the space $J^{k+1}$, cf.\ \cite{KLV,KL1}), 
$a_i=dx_i(X)$ for $(x_1,x_2,x_3,x_4)=(t,x,y,z)$, and
$\phi_p, \phi_q, \phi_r$ are the generating functions for $X$, i.e.\ 
$\phi_p = \omega_p(X), \phi_q = \omega_q(X), \phi_r=\omega_r(X)$ where
\begin{align*}
  \omega_p &= dp-p_t dt-p_x dx-p_y dy-p_z dz, \\
  \omega_q &= dq-q_t dt-q_x dx-q_y dy-q_z dz, \\
  \omega_r &= dr-r_t dt-r_x dx-r_y dy-r_z dz
\end{align*}
In the case of $X_5(m,n)$, the generating functions are given by
 $$
\phi_p = - p_y z^m t^n, \
\phi_q = -\tfrac{n}{2} z^m t^{n-1} - q_y z^m t^n, \
\phi_r = -m z^{m-1} t^n - r_y z^m t^n .
 $$
We see that the restriction of $X_5(m,n)^{(k)}$ to the fiber over $0 \in\mathcal{C}_M^\text{PR}$
is nonzero only when  $m+n \leq k+1$. Hence we can parametrize $\langle X_5 (m,n)\rangle^{(k)}$ by $J_0^{k+1} (\mathbb R^2(t,z), \mathbb R(e))$, and by extending this argument to the whole symmetry algebra we get
(the vector fields $X_k(m,n)$ for $k=1,\dots,4$, are defined similarly to the vector field $X_5(m,n)$ by simply
substituting $a=z^mt^n$ etc into the formulae of Theorem \ref{Thm:Symm})
\begin{align*}
\sym^{(k)} &= \langle X_1(m,n), X_2(m,n),X_4(m,n),X_5(m,n)\rangle^{(k)} \oplus \langle X_3(m,n) \rangle^{(k)} \\ &=J_0^{k+1}(\mathbb R^2(t,z), \mathbb R^4(a,b,d,e)) \times J_0^k(\mathbb R^2(t,z),\mathbb R(c)).
\end{align*}

Using formula (\ref{eq:prolongation}) we verify that the Lie algebra $\sym^{(k)}$ acts freely
on $\mathcal{SDE}_k$ for $k \geq 3$, whence
\begin{align*}
\dim \mathcal O_k &= \dim\left(J_0^{k+1}(\mathbb R^2, \mathbb R^4) \times J_0^k(\mathbb R^2,\mathbb R)\right) \\&= 4\dim \left(J_0^{k+1}(\mathbb R^2, \mathbb R)\right)+\dim\left( J_0^k(\mathbb R^2,\mathbb R)\right) \\
&=4 \binom{k+3}{2}+\binom{k+2}{2}= \frac{(k+2)(5k+13)}{2}.
\end{align*}

\subsection{Counting the differential invariants}

The number $s_k$ of differential invariants of order $k$ (as before, this is $\op{trdeg}\mathfrak{F}_k$)
is equal to the codimension of a generic orbit of $\sym$ on $\mathcal{SDE}_k$. For the lowest orders, we have $s_0=s_1=0$ and $s_2= \dim \mathcal{SDE}_2-\dim \mathcal O_2 = 46-42=4$. For higher jet-orders, the number of invariants of order $k$ is given by \[s_k = \text{codim} \mathcal O_k = \dim \mathcal{SDE}_k - \dim \mathcal O_k = k^3+2k^2-5k-6, \quad k\geq 3.\]

The number of differential invariants of ``pure order'' $k$ is then given by $H(k)=s_k-s_{k-1}$.
The Poincar\'e function $P(z)=\sum_{k=0}^\infty H(k)z^k$ can now easily be computed, and we conclude:
\begin{theorem}
The Hilbert polynomial for the action of $\sym$ on $\mathcal{SDE}$ is
\begin{equation*}
H(k)=\left\{\begin{array}{ll}
0 & \text{ for }k<2,\\
4 & \text{ for }k=2,\\
20 & \text{ for }k=3,\\
3k^2+k-6 & \text{ for }k>3.
\end{array}\right.
\end{equation*} 	
The corresponding Poincar\'e function is equal to
\[P(z) = \frac{2 z^2 (2+4z-z^2-4z^3+2z^4)}{(1-z)^3}.\]
\end{theorem}
Notice that $H(k)$ in this statement has the same leading term as $H(k)$ in Theorem \ref{HandP} for $k>3$.

The following table summarizes the counting results from the last two subsections for low order $k$.

\begin{table}[h]
\def\arraystretch{1.1}
\begin{tabular}{|l||r|r|r|r|r|r|r|r|r|}
\hline
$k$ & 0 & 1 & 2 & 3 & 4 & 5 & 6 & 7 & \dots \\\hline \hline
$\dim \mathcal{SDE}_k$ & 7 & 19 & 46 & 94 & 169 & 277 & 424 & 616 & \dots \\ \hline
$\dim \mathcal O_k$ & 7 & 19 & 42 & 70 & 99 & 133 & 172 & 216 & \dots \\ \hline
$\op{codim}\mathcal O_k$ & 0 & 0 &  4 & 24 & 70 & 144 & 252 & 400 & \dots \\ \hline
$H(k) $ & 0 & 0 & 4 & 20 & 46 & 74 & 108 & 148  & \dots \\ \hline
\end{tabular}
\end{table}

\section{The invariants of $\mathcal{SDE}$ and the quotient equation}\label{S7}


From the global Lie-Tresse theorem \cite{KL2} and Theorem \ref{Algebraicity} it follows that there exist
rational differential invariants of $\g$-action (or $\mathcal{G}$-action) on $\mathcal{SDE}$
that separate generic orbits.

\subsection{Invariants of the second order}
There are four independent differential invariants of the second order:
	\begin{align*}
	I_1 &= \frac{1}{K} \left(	2\,p_{xy}q_{xx}+p_{{yy}}r_{xx}+4\,q_{xy}^{2}+2\,q_{xy}r_{yy}+2\,q_{yy}r_{xy}+r_{yy}^{2}\right)
	\\
	I_2 &= \frac{1}{K^3}  \big( p_{xy}q_{xx}r_{yy}-p_{xy}q_{yy}r_{xx}-p_{yy}q_{xx}r_{xy}+p_{yy}q_{xy}r_{xx} \\ &\quad\; +2\,q_{xy}^{2}r_
	{yy}-2\,q_{xy}q_{yy}r_{xy}+q_{xy}r_{yy}^{2}-q_{{3,
	3}}r_{xy}r_{yy} \big) ^{2}
    \\
	I_3 &= \frac{1}{K^3} \big(  \left(  \left(2\,r_{xy} -2\,q_{xx}
		 \right) p_{xy}+4\,p_{yy}r_{xx}+2\,q_{yy} \left( q_{xx}
		-r_{xy} \right)  \right) q_{xy}\\&\quad\; -4\,q_{xy}^{3}+p_{xy}^{2}r_{xx}-2\,p_{xy}q_{yy}r_{xx}+ \left( q_{xx}-r_{xy} \right) ^{2}p_{yy}+q_{yy}^{2}r_{xx} \big) ^{2}
	\end{align*}
	\begin{align*}
	I_4 &= \frac{1}{K^2} \big(\left( -12\,p_{xy
	}r_{xy}-6\,p_{yy}r_{xx}-12\,q_{yy}q_{xx}+12\,q_{xy}
	^{2} \right) r_{yy}^{2} -3\,r_{yy}^{4}\\& \quad \; + \big( \left(24\,p_{xy} \left( q_{xx}-r_{xy} \right) -12
	\,p_{yy}r_{xx}-24\, q_{{
		yy}} \left( q_{xx}+r_{xy} \right)   \right) q_{xy}\\& \quad \; +48\,q_{xy}^{3}+ 12\, \left(r_{xy}^{2} -q_{xx}^{2}\right) p_{yy}
	+ 12 \,\left( q_{yy}^{2} -p_{xy}^{2}\right) r_{xx}
	 \big) r_{yy}\\& \quad \; + 24\,\left( r_{xy} \left( q_{xx}+r_{
	 	xy} \right) p_{yy}+q_{yy}r_{xx} \left( p_{xy}+q_{yy} \right)  \right) q_{xy}-12\,q_{xy}r_{yy}^{3} \\&\quad \; +3\, \left( 4\, p_{xy}r_{xy}-p_
	{yy}r_{xx} \right)  \left( p_{yy}r_{xx}-4\,q_{yy}q_{xx} \right) \\& \quad \; -24\, \left( p_{yy}r_{xx}+2 \, q_{yy}r_{xy} \right) q_{xy}^{2}
	  \big)
	\end{align*}
where
 $$
K=2\,p_{xy}r_{xy}-p_{yy}r_{xx}+2\,q_{xx}q_{yy}-2\,q_{xy}^{2}+2\,q_{xy}r_{yy}+r_{yy}^{2}
 $$
is a relative differential invariant.

\subsection{Singular set}
Let $\Sigma_2' \subset \mathcal{SDE}_2$ be the set of points $\theta$ where $\langle X_\theta^{(2)} : X \in \sym\rangle \subset T_\theta (\mathcal{SDE}_2)$ is of dimension less than 42. It's given by
 \[
\Sigma'_2 = \{\theta \in \mathcal{SDE}_2 : \text{rank}\,  (\mathcal A|_\theta) < 4\}
 \]
where
\begin{equation*}
\mathcal A = \left(\begin{matrix} 0 & -2 q_{xy}-2r_{yy} & p_{xy}+q_{yy} & 0 \\
0 & 2 p_{xy}-2 q_{yy} & 2 p_{yy} & p_{yy} \\
4 q_{xy}+r_{yy} & -r_{xx} & -2q_{xx} & -2q_{xx} \\
-p_{xy}+q_{yy} & q_{xx}-r_{xy} & 0 & -q_{xy} \\
-p_{yy} & 2q_{xy}-r_{yy} & q_{yy} & 0 \\
-2 q_{xx}+2r_{xy} & 0 & -2 r_{xx} & -3 r_{xx} \\
-2 q_{xy}+r_{yy} & r_{xx} & - r_{xy} & -2 r_{xy} \\
-2 q_{yy} & 2 r_{xy} & 0 & -r_{yy}
\end{matrix} \right).
\end{equation*}
This set contains the singular points that can be seen from a local viewpoint on $\mathcal{SDE}_2$, but there may still be some singular (non-closed) orbits of dimension 42. We use the differential invariants $I_i$ to filter out these. Let $\Sigma_3 \subset \mathcal{SDE}_3$ be the set of points where the 4-form \[\hat dI_1\wedge \hat dI_2 \wedge \hat dI_3 \wedge \hat dI_4\]
is not defined or is zero. Here $\hat d$ is the horizontal differential
\[ \hat d f = \D_t (f) dt+\D_x (f)dx+\D_y (f) dy+\D_z (f)dz.\]
This defines the singular sets $\Sigma_k = (\pi_{k,3}|_{\mathcal{SDE}_k})^{-1}(\Sigma_3) \subset \mathcal{SDE}_k$ and $\Sigma_2 = \pi_{3,2}(\Sigma_3)$. The set $\Sigma_2$ of all singular points in $\mathcal{SDE}_2$ 
contains $\Sigma_2'$.

By using Maple, we can easily verify that $\{K=K_1=K_2=K_3=K_4=0\}$ is contained in $\Sigma_2'$, where $K_i$ is the numerator of $I_i$ for $i=1,2,3,4$. Notice also that 2-jets of conformally flat metrics are contained in $ \Sigma_2'$.

\subsection{Invariants of higher orders}
The 1-forms $\hat dI_1, \hat dI_2, \hat dI_3, \hat dI_4$ determine an invariant horizontal coframe on
$\mathcal{SDE}_3 \setminus \Sigma_3$.
The basis elements of the dual frame $\hat{\p}_{I_1},\hat{\p}_{I_2},\hat{\p}_{I_3},\hat{\p}_{I_4}$ are invariant derivations, the Tresse derivatives.
We can rewrite metric (\ref{ASDmetric}) in terms of the invariant coframe:
\begin{equation}
g = \sum G_{ij} \hat d I_i \hat d I_j,\qquad\text{where} \qquad G_{ij} = g(\hat{\p}_{I_i},\hat{\p}_{I_j}). \label{G}
\end{equation}
Since  the $\hat d I_i$ are invariant, and $[g]$ is invariant, the map \[\hat G = [G_{11}:G_{12}:G_{13}:G_{14}:G_{22}:G_{23}:G_{24}:G_{33}:G_{34}:G_{44}] \colon J^3 \to \mathbb RP^9\]
is invariant. Hence the functions $G_{ij}/G_{44}$ are rational scalar differential invariants (of third order). This has been verified in Maple by differentiation of $G_{ij}/G_{44}$ along the elements of $\sym$. It was also checked that these nine invariants are independent. By the principle of $n$-invariants \cite{ALV}, $I_i$ and $G_{ij}/G_{44}$ generate all scalar differential invariants.

\begin{theorem}
The field of rational differential invariants of $\sym$ on $\mathcal{SDE}$ is generated by the differential invariants $I_k, G_{ij}/G_{44}$ and invariant derivations $\hat{\p}_{I_k}$. The differential invariants in this field separate generic orbits in $\mathcal{SDE}_\infty$.
\end{theorem}

\subsection{The quotient equation}
When restricted to a section $g_0$ of $\mathcal{C}_M^\text{PR}$, the functions $G_{ij}$ can be considered as functions of $I_1,I_2,I_3,I_4$. Two such nonsingular sections are equivalent if they determine the same map $\hat G(I_1,I_2,I_3,I_4)$.

The quotient equation $(\mathcal{SDE}_\infty \setminus \Sigma_\infty)/\sym$ is given by
\[*W_{g} = W_g, \quad  \text{ where } \quad  g = \sum G_{ij}(I_1,I_2,I_3,I_4) \hat d I_i \hat d I_j .\]

Here we consider $I_1,\dots,I_4$ as coordinates on $M$. Equivalently, given local coordinates $(x_1,\dots,x_4)$ on $M$ the quotient equation is obtained by adding to $\mathcal{SDE}$ the equations $I_i=x_i$, $1\le i\le4$.


\end{document}